\documentclass[twoside, 12pt]{article}
\usepackage{amssymb, amsmath, mathrsfs, amsthm}
\usepackage{graphicx}
\usepackage{color}
\usepackage[top=1in, bottom=1in, left=1in, right=1in]{geometry}
\usepackage{float, caption, subcaption}
\usepackage{amsfonts,amsmath,amssymb,amsthm}
\usepackage{graphicx}
\usepackage[linesnumbered,ruled,vlined]{algorithm2e} 
\usepackage{fixmath}
\usepackage{hyperref}
\usepackage{tikz}
\usepackage{makecell} 

\newtheorem{theo}{Theorem}[section]
\newtheorem{prop}[theo]{Proposition}

\newcommand{\old}[1]{{}}

\usepackage{array}
\newcolumntype{L}[1]{>{\raggedright\let\newline\\\arraybackslash\hspace{0pt}}b{#1}}
\newcolumntype{C}[1]{>{\centering\let\newline\\\arraybackslash\hspace{0pt}}b{#1}}
\newcolumntype{R}[1]{>{\raggedleft\let\newline\\\arraybackslash\hspace{0pt}}m{#1}}

\providecommand{\keywords}[1]{\textbf{\textit{Keywords---}} #1} 
\newcolumntype{M}[1]{>{\centering\arraybackslash}m{#1}}
\begin{document}

\pagestyle{plain}

\title{An efficient algorithm to test forcibly-connectedness of graphical degree sequences}

\author{
Kai Wang\footnote{Department of Computer Sciences,
Georgia Southern University,
Statesboro, GA 30460, USA 
\tt{kwang@georgiasouthern.edu}}
}
\maketitle

\begin{abstract}
We present an algorithm to test whether a given graphical degree sequence is forcibly connected or not
and prove its correctness. We also outline the extensions of the algorithm to test
whether a given graphical degree sequence is forcibly $k$-connected or not for every fixed $k\ge 2$.
We show through experimental evaluations that the algorithm is efficient on average, though
its worst case run time is probably exponential. We also adapt Ruskey et al's classic algorithm to enumerate zero-free
graphical degree sequences of length $n$ and Barnes and Savage's classic algorithm
to enumerate graphical partitions of even integer $n$ by incorporating our testing algorithm into theirs and then obtain some
enumerative results about forcibly connected graphical degree sequences of given length $n$ and forcibly connected graphical
partitions of given even integer $n$. Based on these enumerative results we make some conjectures such as: when $n$ is large,
(1) almost all zero-free graphical degree sequences of length $n$ are forcibly connected; (2) almost none of the
graphical partitions of even $n$ are forcibly connected.
\end{abstract}

\section{Introduction}
A graphical degree sequence of finite length $n$ is a non-increasing sequence of non-negative integers $d_1\ge d_2 \ge \cdots \ge d_n$
such that it is the vertex degree sequence of some simple graph (i.e. a finite undirected graph without loops or multiple edges).
Given an arbitrary non-increasing sequence of non-negative integers $a_1\ge a_2 \ge \cdots \ge a_n$, it is easy to test
whether it is a graphical degree sequence by using the Erd{\H{o}}s-Gallai criterion \cite{ErdosCallai1960} or
the Havel-Hakimi algorithm \cite{Havel1955,Hakimi1962}. Seven equivalent criteria to characterize graphical degree sequences
are summarized by Sierksma and Hoogeveen \cite{Sierksma1991}. The notion of partition of an integer is well known
in number theory, and is defined to be a non-increasing sequence of positive integers whose sum is the given integer.
An integer partition is called a graphical partition if it is the vertex degree sequence of some simple graph.
Essentially a zero-free graphical degree sequence and a graphical partition are the same thing.

It is often interesting to observe the properties of all the graphs having the same vertex degree sequence. A graph $G$ with
degree sequence $\mathbf{d}=(d_1\ge d_2 \ge \cdots \ge d_n)$ is called a realization of $\mathbf{d}$. Let P be any property of graphs
(e.g. being bipartite, connected, planar, triangle-free, Hamiltonian, etc). A degree sequence $\mathbf{d}$ is called \textit{potentially}
P-graphic if it has at least one realization having the property P and \textit{forcibly} P-graphic if all its realizations have the property P
\cite{Rao1981}. In this paper we only consider the property of $k$-connectedness ($k\ge 1$ is fixed). Wang and Cleitman
\cite{WangKleitman1973} give a simple characterization of potentially $k$-connected graphical degree sequences of length $n$,
through which we can easily test whether a given graphical degree sequence is potentially connected. However, to the best of our
knowledge no simple characterization of forcibly $k$-connected graphical degree sequences has been found so far and no algorithm
has been published to test whether a given graphical degree sequence is forcibly connected or forcibly $k$-connected with given $k$.
Some sufficient (but unnecessary) conditions are known for a graphical degree sequence
to be forcibly connected or forcibly $k$-connected \cite{chartrand_kapoor_kronk_1968,BOESCH1974,CHOUDUM1991}.

In the rest of this paper we will present a straight-forward algorithm to characterize forcibly connected graphical degree sequences
and outline the extensions of the algorithm to test forcibly $k$-connectedness of graphical degree sequences for fixed $k\ge 2$.
We will demonstrate the efficiency of the algorithm through some computational
experiments and then present some enumerative results regarding forcibly connected graphical degree sequences
of given length $n$ and forcibly connected graphical partitions of given even integer $n$. 
Base on the observations on these available enumerative results we make some conjectures about the relative asymptotic behavior
of considered functions and the unimodality of certain associated integer sequences.

\section{The testing algorithm}
\subsection{Preliminaries}
Based on a result of Wang and Cleitman \cite{WangKleitman1973}, a graphical degree sequence
$d_1\ge d_2 \ge \cdots \ge d_n$ is potentially $k$-connected if and only if $d_n\ge k$ and
$\sum_{i=1}^{n}d_i \ge 2n-2\binom{k}{2}-2+2\sum_{i=1}^{k-1}d_i$. Taking $k=1$, we get that
a zero-free graphical degree sequence $d_1\ge d_2 \ge \cdots \ge d_n$ is potentially connected if and only if
$\sum_{i=1}^{n}d_i \ge 2n-2$.

Note that any graphical degree sequence with a 0 in it can be neither potentially nor forcibly connected. We will design an algorithm to
test whether a zero-free graphical degree sequence $\mathbf{d}$ is forcibly connected based on the simple observation that
$\mathbf{d}$ is forcibly connected if and only if it is not potentially disconnected, i.e., it does not have any disconnected realization.
Equivalently we need to test whether $\mathbf{d}$ can be decomposed into two sub graphical degree sequences. For example,
3,3,3,3,2,2,2 is a potentially connected graphical degree sequence of length 7. It is not forcibly connected
since it can be decomposed into two sub graphical degree sequences 3,3,3,3 and 2,2,2. Note also that when a graphical degree
sequence can be decomposed into two sub graphical degree sequences, the terms in each sub sequence need not be consecutive
in the original sequence. For example, the graphical degree sequence 4,4,3,3,3,2,2,2,1 can be decomposed into two sub
graphical degree sequences 4,4,3,3,3,1 and 2,2,2 or into 4,4,3,3,2 and 3,2,2,1. We say that the graphical degree sequence
3,3,3,3,2,2,2 has a \textit{natural} decomposition because it has a decomposition in which the terms in  each sub sequence
are consecutive in the original sequence. The graphical degree sequence 4,4,3,3,3,2,2,2,1 is not forcibly connected but
does not have a natural decomposition. On the other hand, the graphical degree sequence 6,6,6,5,5,5,5,4,4 is forcibly
connected since there is no way to decompose it into two sub graphical degree sequences.

\subsection{Pseudo-code and the proof of its correctness}
In this section we will present the pseudo-code of our Algorithm \ref{alg:fc}
to test forcibly-connectedness of a given zero-free graphical
degree sequence. We then give a proof why it correctly identifies such graphical degree sequences.

We assume the input is a zero-free
graphical degree sequence already sorted in non-increasing order. In case an input that does not satisfy this condition
is given, we can still easily test whether it is graphical by the Erd{\H{o}}s-Gallai criterion \cite{ErdosCallai1960}
or the Havel-Hakimi algorithm \cite{Havel1955,Hakimi1962}. The output will be \textit{True} if the input is forcibly connected
and \textit{False} otherwise. The output can also include a way to decompose the input in case it is not forcibly connected
and such a decomposition is desired.

\begin{algorithm}[h]
	\KwIn{A zero-free graphical degree sequence $\mathbf{d}=(d_1\ge d_2 \ge \cdots \ge d_n)$}
	\KwOut{\textit{True} or \textit{False}, indicating whether $\mathbf{d}$ is forcibly connected or not}
	\uIf{$d_1\ge n-2$ \textbf{or} $d_n\ge\lfloor n/2 \rfloor$}{
		\Return{\textit{True}}
	}
	\uIf{$d_1= d_n$}{
		\Return{\textit{False}}
	}
	$s_u \gets \max\{s:  s<n-d_{s+1}\}$; \tcp{$2\le s_u\le n-d_n-1$} \
	\uIf{there exists an $s$ such that $d_1+1\le s\le s_u$ and $\mathbf{d_1}=(d_1\ge d_2 \ge \cdots \ge d_s)$ and
		$\mathbf{d_2}=(d_{s+1}\ge d_{s+2} \ge \cdots \ge d_n)$ are both graphical}{
		\Return{\textit{False}}
	}
	\For{$l \gets d_n+1$ \textbf{to} $\min \{\lfloor n/2 \rfloor,n-d_1-1\}$} {
		\uIf{$d_{n+1-l}< l$}{
			$m \gets \min\{i:  d_i<l\}$; \tcp{$1\le m\le n-l+1$} \
			\uIf{$l\le n-m$}{
				Form all candidate decompositions of $\mathbf{d}$ into $\mathbf{s_1}$ and $\mathbf{s_2}$ such that $\mathbf{s_1}$ is
					taken from $\mathbf{d_L}=(d_{m}\ge d_{m+1} \ge \cdots \ge d_n)$ of length $l$ and $\mathbf{s_2=d-s_1}$ is of length $n-l$ and both with even sum.
					\textbf{If} both $\mathbf{s_1}$ and $\mathbf{s_2}$ are graphical, \Return{\textit{False}}\
			}
		}
	}
	\Return{\textit{True}}
	\caption{Pseudo-code to test forcibly-connectedness of a graphical degree sequence}
\label{alg:fc}
\end{algorithm}

Now we show why Algorithm \ref{alg:fc} correctly identifies whether $\mathbf{d}$ is forcibly connected or not.
The conditional test on line 1 works as follows.
\begin{itemize}
\item If $d_1\ge n-2$, then in any realization $G$ of $\mathbf{d}$
the vertex $v_1$ with degree $d_1$
will be in a connected component with at least $n-1$ vertices, leaving at most 1 vertex to be in any other connected component
should the input $\mathbf{d}$ be non forcibly connected. However, a graph with a single vertex has the degree sequence 0, which contradicts the
assumption that $\mathbf{d}$ is zero-free. Thus in this case $\mathbf{d}$ must be forcibly connected.
\item If $d_n\ge\lfloor n/2 \rfloor$, then the vertex $v_n$ with degree $d_n$ will be in a
connected component with at least $1+\lfloor n/2 \rfloor$ vertices. Should the input $\mathbf{d}$ be non forcibly connected, there will be another
connected component not containing $v_n$ and also having at least $1+\lfloor n/2 \rfloor$ vertices since each vertex in that connected
component has degree at least $d_n\ge\lfloor n/2 \rfloor$. This will result in a realization with at least $2+2\lfloor n/2 \rfloor >n$
vertices, a contradiction. Thus in this case $\mathbf{d}$ must also be forcibly connected.
\end{itemize}

The conditional test on line 3 works as follows. Let $d_1=d_n=d$.
\begin{itemize}
\item If $n$ is odd, then $d$ must be even since the
degree sum of a graphical degree sequence must be even. When we reach line 3, we must have $d<\lfloor n/2 \rfloor$. Now $\mathbf{d}$
can be decomposed into two sub graphical degree sequences of length $\frac{n-1}{2}$ and $\frac{n+1}{2}$ respectively since
$d<\frac{n-1}{2}=\lfloor n/2 \rfloor$. Thus it is not forcibly connected.
\item If $n$ is even, we consider two cases.

Case (A): $n/2$ is even, i.e. $n\equiv 0$ mod 4. In this case $\mathbf{d}$ can be decomposed into two graphical degree
sequences of length $n/2$ since $d<\lfloor n/2 \rfloor=n/2$. Thus it is not forcibly connected.

Case (B): $n/2$ is odd, i.e. $n\equiv 2$ mod 4. Further consider two sub cases.
(B1): if $d$ is odd, then $\mathbf{d}$ can be decomposed into two graphical degree sequences
of length $n/2-1$ and $n/2+1$ respectively since $d<n/2-1$ as a result of $d$ and $n/2$ being both odd and $d<n/2$.
(B2): if $d$ is even, then $\mathbf{d}$ can be decomposed into two graphical degree sequences of length $n/2$ since $d<n/2$.
Thus in this case $\mathbf{d}$ is not forcibly connected.
\end{itemize}

Lines 5 to 7 try to find if $\mathbf{d}$ can be decomposed into two sub graphical degree sequences such
that each sub sequence contains terms consecutive in the original sequence $\mathbf{d}$, i.e. if the
input $\mathbf{d}$ has a \textit{natural} decomposition. For each given $s$, the two sub sequences $\mathbf{d_1}=(d_1\ge d_2 \ge \cdots \ge d_s)$
and $\mathbf{d_2}=(d_{s+1}\ge d_{s+2} \ge \cdots \ge d_n)$ can be tested whether they are graphical by utilizing a linear time algorithm
\cite{Ivanyi2013} that is equivalent to the Erd{\H{o}}s-Gallai criterion. The smallest $s$ that need to be tested is $d_1+1$ since
$d_1$ can only be in a graphical degree sequence of length at least $d_1+1$ and $\mathbf{d_1}$ has length $s$.
The largest $s$ that need to be tested is at most
$n-d_n-1$ since $d_n$ can only be in a graphical degree sequence of length at least $d_n+1$ and $\mathbf{d_2}$ has length
$n-s$. Actually the upper bound of the tested $s$ can be chosen to be at most the largest $s$ such that $s<n-d_{s+1}$ since
$d_{s+1}$ can only be in a graphical degree sequence of length at least $d_{s+1}+1$. Let $s_u$ be the largest integer that satisfies
this inequality. Note $s_u\ge 2$ since $s=2$ satisfies the inequality at the point of line 5. Also note that $s_u\le n-d_n-1$ because
if $s_u\ge n-d_n$ then $n-d_n\le s_u<n-d_{s_u+1}$, which leads to $d_{s_u+1}<d_n$, a contradiction. Therefore the upper bound
of tested $s$ is chosen to be $s_u$. Additional data structures can be
maintained to skip the tests of those $s$ for which each of the two sub sequences $\mathbf{d_1}$ and $\mathbf{d_2}$ has odd sum.
Clearly a necessary condition for the input $\mathbf{d}$ to have a natural decomposition is $s_u\ge d_1+1$.
A weaker necessary condition easier to check is $n-d_n-1\ge d_1+1$, i.e. $d_1+d_n\le n-2$.

The \textbf{for} loop starting from line 8 is to test whether the input $\mathbf{d}$ can be decomposed into two sub
graphical degree sequences of length $l$ and $n-l$ respectively, whether the decomposition is natural or not.
At first glance we need to test the range of $l$ in $2\le l\le n-2$ since the shortest zero-free graphical degree sequence
has length 2. By symmetry we do not need to test those $l$ beyond $\lfloor n/2 \rfloor$. Actually we only need to test
the range of $l$ from $d_n+1$ to $\min \{\lfloor n/2 \rfloor,n-d_1-1\}$.
We can start the loop with $l=d_n+1$ since, should the input $\mathbf{d}$ be decomposable,
$d_n$ must be in a sub graphical degree sequence of length at least $d_n+1$ and the other sub graphical
degree sequence not containing $d_n$ must also be of length at least $d_n+1$ due to all its terms being at least $d_n$.
There is no need to test those $l>n-d_1-1$ since, should the input $\mathbf{d}$ be decomposable, $d_1$ must be in a
sub graphical sequence of length at least $d_1+1$ and the other sub graphical sequence not containing $d_1$
must have length at most $n-d_1-1$. 

The condition tested on line 9 ($d_{n+1-l}< l$) is necessary for $\mathbf{d}$ to be decomposable
into two sub graphical degree sequences of length $l$ and $n-l$ respectively. A zero-free graphical degree
sequence of length $l$ must have all its terms less than $l$. If $\mathbf{d}$ is decomposable into two sub graphical degree sequences of length
$l$ and $n-l$ respectively, $\mathbf{d}$ must have at least $l$ terms less than $l$ and $n-l$ terms less than $n-l$. Therefore, the $l$ smallest
terms of $\mathbf{d}$ ($d_{n-l+1}\ge d_{n-l+2}\ge \cdots \ge d_n$) must be all less than $l$ and the $n-l$ smallest
terms of $\mathbf{d}$ ($d_{l+1}\ge d_{l+2}\ge \cdots \ge d_n$) must be all less than $n-l$. These translate to the necessary conditions
$d_{n-l+1}< l$ and $d_{l+1}<n-l$ for $\mathbf{d}$ to be decomposable. The condition $d_{l+1}<n-l$ has already
been satisfied since $d_1<n-l$ based on the loop range of $l$ on line 8.

Lines 10 to 12 first find out \textit{the} sub sequence $\mathbf{d_L}$ of $\mathbf{d}$ consisting \textit{exactly} of those terms less than $l$
and then exhaustively enumerate all sub sequences $\mathbf{s_1}$ of $\mathbf{d_L}$ with length $l$  and even sum, trying to find a valid
decomposition of $\mathbf{d}$ into $\mathbf{s_1}$ and $\mathbf{s_2=d-s_1}$ with length $n-l$, consisting of the terms of $\mathbf{d}$ not
in $\mathbf{s_1}$. Note that the $l$ terms of $\mathbf{s_1}$ need not be consecutive in $\mathbf{d_L}$. The motivation for the
construction of $m$ and $\mathbf{d_L}=(d_{m}\ge d_{m+1} \ge \cdots \ge d_n)$ is that, should the input $\mathbf{d}$ be decomposable
into two sub graphical degree sequences of length $l$ and $n-l$ respectively, the sub graphical degree sequence with length $l$ must
have all its terms coming from $\mathbf{d_L}$. For each such sub sequence $\mathbf{s_1}$ of $\mathbf{d_L}$ with length $l$
(we can always choose such an $\mathbf{s_1}$ since $\mathbf{d_L}$ has length $n-m+1\ge l$ due to the definition of $m$ on line 10),
let the remaining terms of $\mathbf{d}$ form a sub sequence $\mathbf{s_2}=\mathbf{d}-\mathbf{s_1}$ of length $n-l$.
If both $\mathbf{s_1}$ and $\mathbf{s_2}$ are graphical
degree sequences, then the input $\mathbf{d}$ is not forcibly connected since we have found a valid decomposition of $\mathbf{d}$ into
$\mathbf{s_1}$ and $\mathbf{s_2}$ and we may return \textit{False} on line 12.
The conditional test on line 11 ($l\le n-m$) is added because at this point we know $\mathbf{d}$
cannot be naturally decomposed and we can therefore exclude the consideration of $l=n-m+1$ since under this condition
there is only one possible choice of $\mathbf{s_1}$ from $\mathbf{d_L}$ and consequently
only one possible decomposition of $\mathbf{d}$ into two sub sequences of length $l$ and $n-l$ respectively, which
is also a natural decomposition. If we remove the natural decomposition test from lines 5 to 7 and also remove the conditional
test on line 11, the algorithm would obviously still be correct. If in the \textbf{for} loop from lines 8 to 13 we never return \textit{False}
on line 12, this means there is no way to decompose the input $\mathbf{d}$ into two sub
graphical degree sequences whatsoever and we should return \textit{True} on line 14. If we return \textit{False} on line 4, 7, or 12
then a valid decomposition can also be returned if desired.

Later we will show that there is a computable threshold $M(n)$ given the length $n$ of the input $\mathbf{d}$ such that if $d_1$ is below
this threshold the algorithm can immediately return \textit{False} without any exhaustive enumerations.
However, our computational experiences suggest that if the input satisfies
$d_1<M(n)$ then Algorithm \ref{alg:fc} already runs fast and it might not be worthwhile to add the computation
of the additional threshold $M(n)$ into it.

\subsection{Extensions of the algorithm}
In this section we show how to extend Algorithm \ref{alg:fc} to perform additional tasks such as listing all
possible decompositions of a graphical degree sequence and testing forcibly $k$-connectedness of a graphical
degree sequence for fixed $k\ge 2$.
\subsubsection{Enumeration of all possible decompositions}
Algorithm \ref{alg:fc} can be easily extended to give all possible decompositions of the input $\mathbf{d}$ into two
sub graphical degree sequences in case it is not forcibly connected. We simply need to report a valid decomposition found
on line 3, 6 and 12 and continue without returning
\textit{False} immediately. Such an enumerative algorithm to find all valid decompositions of the input $\mathbf{d}$
can be useful when we want to explore the possible realizations of $\mathbf{d}$ and their properties.

\subsubsection{Testing forcibly $k$-connectedness of $\mathbf{d}$ when $k\ge 2$}
\label{sec:extensionAlgs}
It is also possible to extend Algorithm \ref{alg:fc} to decide whether a given graphical degree sequence $\mathbf{d}$
is forcibly biconnected or not. We know that
a connected graph is biconnected (non-separable) if and only if it does not have a cut vertex. This characterization leads us to
observe that if in any forcibly connected graphical degree sequence $\mathbf{d}$ the removal of any term $d_i$ and the
reduction of some collection $\mathbf{d_S}$ of $d_i$ elements from the remaining sequence $\mathbf{d}-\{d_i\}$ by 1
each results in a non forcibly connected graphical
degree sequence $\mathbf{d'}$ then $\mathbf{d}$ is not forcibly biconnected. If no such term and a corresponding
collection of elements from the remaining sequence can be found whose removal/reduction results in a non forcibly connected graphical
degree sequence, then $\mathbf{d}$ is forcibly biconnected. We give a pseudo-code framework in Algorithm \ref{algf2}
to decide whether a given graphical degree sequence $\mathbf{d}$ is forcibly biconnected or not. To simplify our
description, we call the above mentioned combination of removal/reduction operations a \textit{generalized Havel-Hakimi} (GHH) operation,
notationally $\mathbf{d'}=GHH(\mathbf{d},d_i,\mathbf{d_S})$. We remark that if the $\mathbf{d'}$ obtained on line 4 of Algorithm \ref{algf2}
is not a graphical degree sequence then the condition on line 5 is not satisfied and the algorithm will not return \textit{False}
at the moment.

\begin{algorithm}[h]
	\KwIn{A zero-free graphical degree sequence $\mathbf{d}=(d_1\ge d_2 \ge \cdots \ge d_n)$}
	\KwOut{\textit{True} or \textit{False}, indicating whether $\mathbf{d}$ is forcibly biconnected or not}
	\uIf{$\mathbf{d}$ is not potentially biconnected or forcibly connected}{
		\Return{\textit{False}}
	}
	\For{each $d_i$ and each collection $\mathbf{d_S}$ of size $d_i$ from $\mathbf{d}-\{d_i\}$} {
		$\mathbf{d'} \gets GHH(\mathbf{d},d_i,\mathbf{d_S})$\;
		\uIf{$\mathbf{d'}$ is a non forcibly connected graphical degree sequence}{
			\Return{\textit{False}}
		}
	}
	\Return{\textit{True}}
	\caption{Pseudo-code to test whether a graphical degree sequence is forcibly biconnected (See text for the description of GHH operation)}
	\label{algf2}
\end{algorithm}

Similarly we can test whether a given graphical degree sequence $\mathbf{d}$ is forcibly $k$-connected or not for $k\ge 3$ iteratively
as long as we already have a procedure to test whether a graphical degree sequence is forcibly $(k-1)$-connected or not. Suppose we
already know an input $\mathbf{d}$ is potentially $k$-connected and forcibly $(k-1)$-connected. We can proceed to choose a term
$d_i$ and a collection $\mathbf{d_S}$ of size $d_i$ from the remaining sequence $\mathbf{d}-\{d_i\}$ and perform a GHH operation
on $\mathbf{d}$. If the resulting sequence $\mathbf{d'}$ is a graphical degree sequence and is non forcibly $(k-1)$-connected,
then $\mathbf{d}$ is not forcibly $k$-connected.
If no such term and a corresponding collection of elements from the remaining sequence can be found whereby a GHH operation
can be performed on $\mathbf{d}$ to
result in a non forcibly $(k-1)$-connected graphical degree sequence, then $\mathbf{d}$ is forcibly $k$-connected.
We give a pseudo-code framework in Algorithm \ref{algfk}
to decide whether a given graphical degree sequence $\mathbf{d}$ is forcibly $k$-connected or not.

\begin{algorithm}[h]
	\KwIn{A zero-free graphical degree sequence $\mathbf{d}=(d_1\ge d_2 \ge \cdots \ge d_n)$ and an integer $k\ge 2$}
	\KwOut{\textit{True} or \textit{False}, indicating whether $\mathbf{d}$ is forcibly $k$-connected or not}
	\uIf{$\mathbf{d}$ is not potentially $k$-connected or forcibly $(k-1)$-connected}{
		\Return{\textit{False}}
	}
	\For{each $d_i$ and each collection $\mathbf{d_S}$ of size $d_i$ from $\mathbf{d}-\{d_i\}$} {
		$\mathbf{d'} \gets GHH(\mathbf{d},d_i,\mathbf{d_S})$\;
		\uIf{$\mathbf{d'}$ is a non forcibly $(k-1)$-connected graphical degree sequence}{
			\Return{\textit{False}}
		}
	}
	\Return{\textit{True}}
	\caption{Pseudo-code to test whether a graphical degree sequence is forcibly $k$-connected (See text for the description of GHH operation)}
	\label{algfk}
\end{algorithm}


\section{Complexity analysis}
We conjecture that Algorithm \ref{alg:fc} runs in time polynomial in $n$ on average. The worst case run time
complexity is probably still exponential in $n$. We are unable to provide a
rigorous proof at this time, but we will later show through experimental evaluations that it runs fast on randomly generated
long graphical degree sequences most of the time.

Now we give a discussion of the run time behavior of Algorithm \ref{alg:fc}.
Observe that lines 1 to 4 take constant time. Lines 5 to 7 take $O(n^2)$ time if we use the linear time algorithm from
\cite{Ivanyi2013} to test whether an integer sequence is graphical. Lines 9 to 11 combined take $O(n)$ time and they are
executed $O(n)$ times. So the overall time complexity is $O(n^2)$ excluding the time on line 12.

Next consider all candidate decompositions of $\mathbf{d}$ into $\mathbf{s_1}$ and $\mathbf{s_2}$ on line 12. The sub
sequence $\mathbf{s_1}$ is taken from $\mathbf{d_L}=(d_{m}\ge d_{m+1} \ge \cdots \ge d_n)$ whose length could be as
large as $n$ and the length $l$ of $\mathbf{s_1}$ could be as large as $\lfloor n/2 \rfloor$. Therefore in the worst case we may have up to
$\binom{n}{n/2}$ candidate decompositions, which could make the run time of Algorithm \ref{alg:fc} exponential in $n$.

A careful implementation of Algorithm \ref{alg:fc} will help reduce running time by noting that $\mathbf{d_L}$ is
a multi-set and provides us an opportunity to avoid duplicate enumerations
of $\mathbf{s_1}$ because different $l$ combinations of the indices $(m,m+1,\cdots,n)$ could produce the same sub sequence
$\mathbf{s_1}$. For this purpose, we can assume
the input $\mathbf{d}$ is also provided in another format $(e_1,f_1), (e_2,f_2),\cdots,(e_q,f_q)$ where $\mathbf{d}$
contains $f_i$ copies of $e_i$ for $i=1,\cdots,q$ and $e_1>e_2>\cdots>e_q>0$. (Clearly $d_1=e_1$ and $d_n=e_q$.)
Now enumerating $\mathbf{s_1}$ of length $l$ from $\mathbf{d_L}$ can be equivalently translated to the following problem of enumerating
all non-negative integer solutions of Equation (\ref{eqn:enum}) subject to constraints (\ref{eqn:constraints}),
\begin{equation} \label{eqn:enum}
\sum_{i=1}^{k}x_i=l,
\end{equation}
\begin{equation} \label{eqn:constraints}
0\le x_i\le f_{q-k+i}, \mbox{ for } i=1,\cdots,k,
\end{equation}
where $k$ is the number of distinct elements in $\mathbf{d_L}=(d_{m}\ge d_{m+1} \ge \cdots \ge d_n)$ which can also be
represented as $(e_{q-k+1},f_{q-k+1}), (e_{q-k+2},f_{q-k+2}),\cdots,(e_q,f_q)$ and $k$ satisfies $k\le q$ and $k\le l-d_n$
since all the elements of $\mathbf{d_L}$ are $<l$ and $\ge d_n$. In this context $m$ and $k$ vary with $l$ as the \textbf{for}
loop from lines 8 to 12 progresses. Each solution of Equation (\ref{eqn:enum}) represents a candidate
choice of $\mathbf{s_1}$ out of $\mathbf{d_L}$ with length $l$ by taking $x_i$ copies of $e_{q-k+i}$. Further improvement
could be achieved by noting the odd terms among $e_{q-k+1},e_{q-k+2},\cdots,e_q$ since we must have an even number
of odd terms in $\mathbf{s_1}$ for it to have even sum. We can categorize the $x_i$ variables of Equation (\ref{eqn:enum})
into two groups based on the parity of the corresponding $e_i$ and enumerate only its solutions having an even sum of
the $x_i$'s belonging to the odd group.

The number of solutions of
Equation (\ref{eqn:enum}) can be exponential in $n$. For example, let $l=n/2,k=n/4$ and let $f_j=4$ for
$j=q-k+1,\cdots,q$. Then the number of solutions of Equation (\ref{eqn:enum}) will be at least $\binom{n/4}{n/8}$
by taking half of all $x_1,\cdots,x_k$ to be 4 and the remaining half to be 0. However in practice we rarely find such a
large number of solutions are actually all enumerated before Algorithm \ref{alg:fc} returns.

To the best of our knowledge, the computational complexity of the decision problem of whether a given graphical degree
sequence is forcibly connected is unknown. The problem is clearly in co-NP since a short certificate
to prove that a given input is not forcibly connected is a valid decomposition of the input sequence. But is it co-NP-hard?
As far as we know, this is an open problem.

The time complexity of the extension Algorithms \ref{algf2} and \ref{algfk} to test whether a given graphical degree
sequence $\mathbf{d}$ is forcibly $k$-connected or not for $k\ge 2$ is apparently exponential due to
the exhaustive enumeration of the candidate collection $\mathbf{d_S}$ of size $d_i$ from the remaining sequence $\mathbf{d}-\{d_i\}$
and the ultimate calls to Algorithm \ref{alg:fc} possibly an exponential number of times.

The computational complexity of the decision problem of whether a given graphical degree
sequence is forcibly $k$-connected ($k\ge 2$) is also unknown to us. Clearly the problem is still in co-NP when $k$ is fixed
as to prove that a graphical degree sequence $\mathbf{d}$ is not forcibly $k$-connected is as easy as using a sequence of $k-1$ certificates
each consisting of a pair ($d_i,\mathbf{d_S}$) and a $k^{th}$ certificate being a valid decomposition to show that the
final resulting sequence is decomposable after a sequence of $k-1$ GHH operations on $\mathbf{d}$,
but we do not know if it is inherently any harder than the decision problem for $k=1$.

\section{Computational results}
In this section we will first present some results on the experimental evaluation on the performance
of Algorithm \ref{alg:fc} on randomly generated long
graphical degree sequences. We will then provide some enumerative results about the number of forcibly connected graphical
degree sequences of given length and the number of forcibly connected graphical partitions of a given even integer. Based on
the available enumerative results we will make some conjectures about the asymptotic
behavior of related functions and the unimodality of certain associated integer sequences.

\subsection{Performance evaluations of algorithm \ref{alg:fc}}
In order to evaluate how efficient Algorithm \ref{alg:fc} is, we aim to generate long testing instances
with length $n$ in the range of thousands and see how Algorithm \ref{alg:fc} performs on these
instances.

Our experimental methodology is as follows. Choose a constant $p_h$ in the range [0.1,0.95] and a constant $p_l$
in the range of [0.001,$\min\{p_h-0.01,0.49\}$] and generate 100 random graphical degree sequences of length $n$ with
largest term around $p_hn$ and smallest term around $p_ln$. Each such graphical degree sequence is generated by first
uniformly random sampling integer partitions with the specified number of parts $n$ and the specified largest part and
smallest part and then accept it as input for Algorithm \ref{alg:fc} if it is a graphical degree sequence.
We run Algorithm \ref{alg:fc} on these random
instances and record the average performance and note the proportion of them that are forcibly connected. Table \ref{tab:chosenphpl}
lists the tested $p_h$ and $p_l$. The largest tested $p_l$ is 0.49 since any graphical degree sequence of length $n$
and smallest term at least $0.5n$ will cause Algorithm \ref{alg:fc} to return \textit{True} on line 2.

\begin{table}[!htb]
	\centering
	\caption{Chosen $p_h$ and $p_l$ in the experimental performance evaluation of Algorithm \ref{alg:fc}.}
	\begin{tabular}{|c|c|}
		\hline
		$p_h$ & $p_l$ \\
		\hline
		\hline
		0.10 & 0.001,0.002,0.003,...,0.01,0.02,0.03,...,0.09 \\ \hline
		0.20 & 0.001,0.002,0.003,...,0.01,0.02,0.03,...,0.19 \\ \hline
		0.30 & 0.001,0.002,0.003,...,0.01,0.02,0.03,...,0.29 \\ \hline
		0.40 & 0.001,0.002,0.003,...,0.01,0.02,0.03,...,0.39 \\ \hline
		0.50 & 0.001,0.002,0.003,...,0.01,0.02,0.03,...,0.49 \\ \hline
		0.55 & 0.001,0.002,0.003,...,0.01,0.02,0.03,...,0.49 \\ \hline
		0.60 & 0.001,0.002,0.003,...,0.01,0.02,0.03,...,0.49 \\ \hline
		0.65 & 0.001,0.002,0.003,...,0.01,0.02,0.03,...,0.49 \\ \hline
		0.70 & 0.001,0.002,0.003,...,0.01,0.02,0.03,...,0.49 \\ \hline
		0.75 & 0.001,0.002,0.003,...,0.01,0.02,0.03,...,0.49 \\ \hline
		0.80 & 0.001,0.002,0.003,...,0.01,0.02,0.03,...,0.49 \\ \hline
		0.85 & 0.001,0.002,0.003,...,0.01,0.02,0.03,...,0.49 \\ \hline
		0.90 & 0.001,0.002,0.003,...,0.01,0.02,0.03,...,0.49 \\ \hline
		0.95 & 0.001,0.002,0.003,...,0.01,0.02,0.03,...,0.49 \\ \hline
	\end{tabular}
	\label{tab:chosenphpl}
\end{table}

\begin{table}[!htb]
	\centering
	\caption{Transition interval $I_t$ of $p_l$ for each $p_h$ (for $n=1000$).}
	\begin{tabular}{|c|c|}
		\hline
		$p_h$ & $I_t$ of $p_l$ \\
		\hline
		\hline
		0.55 & 0.30 to 0.40 \\ \hline
		0.60 & 0.20 to 0.30 \\ \hline
		0.65 & 0.15 to 0.24 \\ \hline
		0.70 & 0.09 to 0.17 \\ \hline
		0.75 & 0.05 to 0.12 \\ \hline
		0.80 & 0.03 to 0.09 \\ \hline
		0.85 & 0.01 to 0.07 \\ \hline
		0.90 & 0.003 to 0.04 \\ \hline
		0.95 & 0.001 to 0.03 \\ \hline
	\end{tabular}
	\label{tab:transitionrange}
\end{table}
We implemented our Algorithm \ref{alg:fc} using C++ and compiled it using g++ with optimization level -O3.
The experimental evaluations are performed on a common Linux workstation.
We summarize our experimental results for the length $n=1000$ as follows.

1. For those instances with $p_h$ in the range from 0.1 to 0.5, Algorithm \ref{alg:fc} always finishes instantly (run
time $< 0.01s$) and all the tested instances are non forcibly connected. This does not necessarily mean that there are
no forcibly connected graphical degree sequences of length $n=1000$ with largest term around $p_hn$ with $p_h$ in this range.
It only suggests that forcibly connected graphical degree sequences are relatively rare in this range.

2. For each $p_h$ in the range from 0.55 to 0.95, we observed a transition interval $I_t$ of $p_l$ for each fixed $p_h$.
See Table \ref{tab:transitionrange} for a list of observed transition intervals. All those instances with
$p_l$ below the range $I_t$ are non forcibly connected and all those instances with $p_l$ above the range $I_t$ are forcibly
connected. Those instances with $p_l$ in the range $I_t$ exhibit the behavior that the proportion of forcibly connected among all
tested 100 instances gradually increases from 0 to 1 as $p_l$ increases in the range $I_t$. For example, based on the results
of Table \ref{tab:transitionrange}, the proportions of forcibly connected graphical degree sequences of length 1000 with largest
term around 850 ($p_h=0.85$) and smallest term below around 10 ($p_l=0.01$) are close to 0. If the smallest term is above
around 70 ($p_l=0.07$) then the proportion is close to 1. When the smallest term is between 10 and 70 ($p_l$ in the range from
0.01 to 0.07) then the proportion transitions from 0 to 1. Again these results should be interpreted as relative
frequency instead of absolute law.

3. Algorithm \ref{alg:fc} is efficient most of the time but encounters bottlenecks at occasions. For $p_h$ from 0.80 to 0.95
and $p_l$ near the lower end of the transition interval $I_t$, Algorithm \ref{alg:fc} does perform poorly on some of the tested
instances with run time from a few seconds to more than a few hours (time out). The exact range of $p_l$ near the lower end of $I_t$
where Algorithm \ref{alg:fc} could perform poorly varies. We observed
that this range of $p_l$ for which the algorithm could perform poorly is quite narrow. For example, when $n=1000$, $p_h=0.9$, this
range of $p_l$ we observed is from 0.001 to 0.01. We observed that the frequency at which the algorithm performs poorly also varies.
We believe that this is because among all possible instances with given length and given largest and smallest terms there is still great
variety in terms of difficulty of testing their property of forcibly connectedness using Algorithm \ref{alg:fc}. In particular, some
instances will trigger the exhaustive behavior of Algorithm \ref{alg:fc} on line 12, making it enumerate a lot of
candidate decompositions without returning.

We have also performed experimental evaluations of Algorithm \ref{alg:fc} for the length $n=2000,3000,...,10000$ without being able
to finish all the same $p_h,p_l$ choices as for $n=1000$ because of shortage of time. The behavior of Algorithm \ref{alg:fc}
on inputs of these longer lengths
is similar to the case of $n=1000$ but with different transition intervals $I_t$ and varied range of $p_l$ near the lower end of the transition
interval $I_t$ for which it could perform poorly.

To sum up, we believe that the average case run time of Algorithm \ref{alg:fc} is polynomial. We estimate that more
than half of all zero-free graphical degree sequences of length $n$ can be tested in constant time on line 1. However, its worst
case run time should be exponential. As mentioned above, the computational complexity of the decision problem itself
is unknown to us.

Currently we have a very rudimentary implementation of Algorithm \ref{algf2} and do not have an implementation of Algorithm
\ref{algfk} for any $k\ge 3$ yet. Algorithm \ref{algf2} can start to encounter bottlenecks for input length $n$ around 40 to 50, which
is much shorter than the input lengths Algorithm \ref{alg:fc} can handle. We suspect that to handle input length $n\ge 100$
when $k=3$ will be very difficult unless significant enhancement to avoid many of those exhaustive enumerations can be introduced.

\subsection{Enumerative results}
In this section we will present some enumerative results related to forcibly connected graphical degree sequences
of given length and forcibly connected graphical partitions of given even integer. We also make some conjectures
based on these enumerative results. For the reader's convenience, we summarize the notations
used in this section in Table \ref{tbl:definitions}.

\begin{table}[!htb]
	\centering
	\caption{Terminology used in this section}
	\begin{tabular}{||c|l||}
		\hline\hline
		Term & Meaning\\
		\hline\hline
		$D(n)$ & number of zero-free graphical sequences of length $n$\\
		\hline
		$D_c(n)$ & number of potentially connected graphical sequences of length $n$ \\
		\hline
		$D_f(n)$ & number of forcibly connected graphical sequences of length $n$ \\
		\hline
		$C_n[N]$ & number of potentially connected graphical degree sequences \\
		& of length $n$ with degree sum $N$ \\
		\hline
		$F_n[N]$ & number of forcibly connected graphical degree sequences \\
		& of length $n$ with degree sum $N$ \\
		\hline
		$L_{n}[j]$ & number of forcibly connected graphical degree sequences of\\
		& length $n$ with largest term $j$ \\
		\hline
		$M(n)$ & minimum largest term in any forcibly connected graphical \\
		& sequence of length $n$ \\
		\hline
		$g(n)$ & number of graphical partitions of even $n$ \\
		\hline
		$g_c(n)$ & number of potentially connected graphical partitions of even $n$ \\
		\hline
		$g_f(n)$ & number of forcibly connected graphical partitions of even $n$ \\
		\hline
		$c_n[j]$ & number of potentially connected graphical partitions of \\
		& even $n$ with $j$ parts \\
		\hline
		$f_n[j]$ & number of forcibly connected graphical partitions of \\
		& even $n$ with $j$ parts \\
		\hline
		$l_n[j]$ & number of forcibly connected graphical partitions of $n$ with largest term $j$\\
		\hline
		$m(n)$ & minimum largest term of forcibly connected graphical partitions of $n$\\
		\hline\hline
	\end{tabular}
	\label{tbl:definitions}
\end{table}

In a previous manuscript \cite{Wang2016} we have presented efficient algorithms for counting the number of graphical degree sequences of length $n$
and the number of graphical degree sequences of $k$-connected graphs with $n$ vertices (or graphical degree sequences of length $n$ that
are potentially $k$-connected). It is proved there that the asymptotic orders of the number $D(n)$ of zero-free graphical degree sequences of
length $n$ and the number $D_c(n)$ of potentially connected graphical degree sequences of length $n$ are equivalent.
That is, $\lim_{n \to \infty} \frac{D_c(n)}{D(n)} = 1$. In order to investigate
how the number $D_f(n)$ of forcibly connected graphical degree sequences of length $n$ grows compared to $D(n)$
we conduct computations to count such graphical degree sequences. We do not have any algorithm that can get
the counting result without actually generating the sequences. The fastest algorithm we know of that can generate
all zero-free graphical degree sequences of length $n$ is from Ruskey et al \cite{Ruskey1994}. We adapted this algorithm
to incorporate the test in Algorithm \ref{alg:fc} to
count those that are forcibly connected. Since $D(n)$ grows as an exponential function of $n$ based on the bounds given by
Burns \cite{Burns2007} ($4^n/(c_1 n) \le D(n) \le 4^n/((\log n)^{c_2} \sqrt{n})$) for all sufficiently large $n$ with $c_1,c_2$
positive constants), it is unlikely to get the value of $D_f(n)$ for large $n$ using an exhaustive generation algorithm.
We only have counting results of $D_f(n)$
for $n$ up to 26 due to the long running time of our implementation. The results together with the proportion
of them in all zero-free graphical degree sequences are listed in Table \ref{tab:enumDf(n)}. From the table it seems
reasonable to conclude that the proportion $D_f(n)/D(n)$ will increase when $n\ge 8$ and it might tend to the limit 1.
\begin{table}[!htb]
	\centering
	\caption{Number of forcibly connected graphical degree sequences of length $n$ and their proportions in zero-free graphical degree sequences of length $n$.}
	\begin{tabular}[htbp]{|c||c|c|c|}
		\hline
		$n$ & $D(n)$ & $D_f(n)$ & $D_f(n)/D(n)$ \\
		\hline
		\hline
		4 & 7  & 6 & 0.857143 \\ \hline
		5 & 20 & 18 & 0.900000 \\ \hline
		6 & 71 & 63 & 0.887324 \\ \hline
		7 & 240 & 216 & 0.900000 \\ \hline
		8 & 871 & 783 & 0.898967 \\ \hline
		9 & 3148 & 2843 & 0.903113 \\ \hline
		10 & 11655 & 10535 & 0.903904 \\ \hline
		11 & 43332 & 39232 & 0.905382 \\ \hline
		12 & 162769 & 147457 & 0.905928 \\ \hline
		13 & 614198 & 556859 & 0.906644 \\ \hline
		14 & 2330537 & 2113982 & 0.907079 \\ \hline
		15 & 8875768 & 8054923 & 0.907518 \\ \hline
		16 & 33924859 & 30799063 & 0.907861 \\ \hline
		17 & 130038230 & 118098443 & 0.908182 \\ \hline
		18 & 499753855 & 454006818 & 0.908461 \\ \hline
		19 & 1924912894 & 1749201100 & 0.908717 \\ \hline
		20 & 7429160296 & 6752721263 & 0.908948 \\ \hline
		21 & 28723877732 & 26114628694 & 0.909161 \\ \hline
		22 & 111236423288 & 101153550972 & 0.909356 \\ \hline
		23 & 431403470222 & 392377497401 & 0.909537 \\ \hline
		24 & 1675316535350 & 1524043284254 & 0.909705 \\ \hline
		25 & 6513837679610 & 5926683351876 & 0.909860 \\ \hline
		26 & 25354842100894 & 23073049582134 & 0.910006 \\ \hline
	\end{tabular}
	\label{tab:enumDf(n)}
\end{table}

\begin{table}[!htb]
	\centering
	\caption{Number of potentially (row $C_7[N]$) and forcibly (row $F_7[N]$) connected graphical degree sequences of length 7 with given degree sum $N$.}
	\begin{tabular}[htbp]{|c||c|c|c|c|c|c|c|c|c|c|c|c|c|c|c|c|}
		\hline
		degree sum $N$ & 12 & 14 & 16 & 18 & 20 & 22 & 24 & 26 & 28 & 30 & 32 & 34 & 36 & 38 & 40 & 42\\ \hline
		\hline
		$C_7[N]$ & 7 & 11 & 15 & 22 & 26 & 29 & 29 & 26 & 23 & 18 & 13 & 8 & 5 & 2 & 1 & 1 \\
		\hline
		$F_7[N]$ & 3 & 5 & 10 & 19 & 25 & 28 & 29 & 26 & 23 & 18 & 13 & 8 & 5 & 2 & 1 & 1 \\ \hline
	\end{tabular}
	\label{tab:enumDfdegsums(7)}
\end{table}

\begin{table}[!htb]
	\centering
	\caption{Number $L_{15}[j]$ of forcibly connected graphical degree sequences of length 15 with given largest term $j$.}
	\begin{tabular}[htbp]{|c||c|c|c|c|c|c|c|c|c|}
		\hline
		largest part $j$ & 14 & 13 & 12 & 11 & 10 & 9 & 8 & 7 & 6 \\ \hline
		\hline
		$L_{15}[j]$ & 3166852 & 2624083 & 1398781 & 600406 & 201128 & 52903 & 9718 & 1031 & 21 \\ \hline
	\end{tabular}
	\label{tab:enumDflargestpart(15)}
\end{table}

Since our adapted algorithm from Ruskey et al \cite{Ruskey1994} for computing $D_f(n)$ actually generates all forcibly
connected graphical degree sequences of length $n$ it is trivial to also output the individual counts based on the degree sum $N$
or the largest degree $\Delta$. That is, we can output the number of forcibly connected graphical degree sequences
of length $n$ with degree sum $N$ or largest term $\Delta$.
In Table \ref{tab:enumDfdegsums(7)} we show itemized potentially and forcibly connected graphical degree sequences
of length 7 based on the degree sum $N$. The counts for $N<12$ are not shown because those counts are all 0.
The highest degree sum is 42 for any graphical degree sequence of length 7.
From the table we see that the individual counts based on the degree sum $N$ that contribute to $D_c(7)$ (row $C_7[N]$)
and $D_f(7)$ (row $F_7[N]$) both form a unimodal sequence. Counts for other degree sequence lengths
from 5 to 26 exhibit similar behavior. Based on the available enumerative results
we find that for any given $n$ the range of even $N$ for $C_n[N]$ and $F_n[N]$ to be nonzero respectively
are exactly the same (between $2n-2$ and $n(n-1)$). In fact, this can be proved as the following
\begin{prop}
	\label{thm:prop_potentialForciblyAgree} An even $N$ has a potentially connected graphical partition with $n$ parts if and only
	if it has a forcibly connected graphical partition with $n$ parts.
\end{prop}
\begin{proof}
	Sufficiency is obvious by definition. In the following we show the necessity.
	
	Suppose an even $N$ has a potentially connected graphical partition with $n$ parts.
	From the Wang and Cleitman characterization \cite{WangKleitman1973} we know that $N$ must be 
	between $2n-2$ and $n(n-1)$ for it to have a potentially connected graphical partition of $n$ parts. Now construct
	a partition $\pi$ of $N$ with $n$ parts as follows. Let the largest part be $n-1$ and let the remaining $n-1$ parts
	be as evenly as possible. That is, let $b=\lfloor \frac{N-n+1}{n-1} \rfloor$ and $a=N-(n-1)(b+1)$. Then the smallest
	$n-1$ parts of $\pi$ consist of $a$ copies of $b+1$ and $n-1-a$ copies of $b$. With $2n-2\le N\le n(n-1)$, we have
	$0<b\le n-1$ and $0\le a<n-1$. Based on the Nash-Williams condition it is easy to verify that $\pi$ is a graphical partition
	of $N$ with $n$ parts and it is forcibly connected since its largest part is $n-1$.
\end{proof}

In Table \ref{tab:enumDflargestpart(15)} we show itemized numbers of forcibly connected graphical degree
sequences of length 15 based on the largest degree. The counts for largest degrees less than 6 are not shown because those
counts are all 0. From the table we can see that the counts decrease with the largest degree.
For other degree sequence lengths from 5 to 26 we observed
similar behavior. The table also indicates that there are no forcibly connected graphical
degree sequences of length 15 with largest degree less than 6. In fact, if we define $M(n)$ to be the minimum largest term
in any forcibly connected graphical sequence of length $n$. This is, $M(n)\doteq$ min\{$\Delta$: $\Delta$ is the largest term
of some forcibly connected graphical degree sequence of length $n$\}. Clearly we have $M(n)\le n/2$ since for even $n$
the sequence $n/2,n/2,\cdots,n/2$ of length $n$ is forcibly connected. We can show a lower bound of $M(n)$ as follows.

\begin{theo}
	\label{thm:MnGrowth}
	For $M(n)$ defined above, we have $M(n)=\Omega(\sqrt{n})$. That is, there is a constant $c>0$ such that $M(n)>c\sqrt{n}$ for all sufficiently
	large $n$.
\end{theo}
\begin{proof}
	For the purpose of deriving a contradiction assume there is a forcibly connected graphical degree sequence
	$\mathbf{\pi}=(d_1\ge d_2 \ge \cdots \ge d_n)$ of length $n$ with the largest term $d_1=M(n)=o(\sqrt{n})$.
	
	Let us first consider the case that $n$ is even. Let $\mathbf{\pi_H}$ be the higher half
	(of length $n/2$) of $\mathbf{\pi}$ and $\mathbf{\pi_L}$ be the lower half
	(of length $n/2$) of $\mathbf{\pi}$. If both $\mathbf{\pi_H}$ and $\mathbf{\pi_L}$ have even sums, then they can shown to be
	both graphical degree sequences based on the Nash-Williams condition \cite{Ruch1979,Rousseau1995,Sierksma1991} as follows.
	Suppose the Durfee square size of $\mathbf{\pi}$ is $s$ where $s\le d_1=M(n)$ by the definition of Durfee square. Since
	$\mathbf{\pi}$ is graphical it satisfies the Nash-Williams condition, which can be represented as $s$ inequalities:
	\[ \sum_{i=1}^{j}(d'_{i}-d_i)\ge j, \mbox{  }j=1,\cdots, s,  \]
	where $d'_1,\cdots,d'_s$ are the largest $s$ parts of the conjugate partition of the partition $\mathbf{\pi}$ ($d'_1=n$).
	Now $\mathbf{\pi}$
	and $\mathbf{\pi_H}$ have the same Durfee square size by our assumption that $s=o(\sqrt{n})$ and they have the same $s$ largest parts.
	Let the $s$ largest parts of the conjugate of $\mathbf{\pi_H}$ be $d''_1,\cdots,d''_s$ with $d''_1=n/2$ by our construction.
	To show that $\mathbf{\pi_H}$ is graphical, we only need to show that the following $s$ inequalities hold:
	\begin{equation} \label{eqn:Nash-WilliamsInequalities}
	\sum_{i=1}^{j}(d''_{i}-d_i)\ge j, \mbox{  }j=1,\cdots, s.
	\end{equation}
	The first of these inequalities $d''_1-d_1=n/2-M(n)\ge 1$ is clearly satisfied since $M(n)=o(\sqrt{n})$.
	We also have the following inequalities,
	\[ d''_j\ge s \mbox{ and } d_j\le M(n), \mbox{  } j=2,\cdots,s, \]
	so we have $d''_j-d_j\ge s-M(n), j=2,\cdots,s$. Even if $d''_j-d_j, j=2,\cdots,s$ are all negative, their sum will be of order
	$o(n)$ since $s\le M(n)=o(\sqrt{n})$. Clearly the $s$ inequalities in (\ref{eqn:Nash-WilliamsInequalities}) are all satisfied since
	$d''_1-d_1=n/2-M(n)$ is of order $\Omega(n)$. This shows that $\mathbf{\pi_H}$ is graphical. By the same argument $\mathbf{\pi_L}$
	is graphical and we have found that $\mathbf{\pi}$ can be decomposed into two sub graphical degree sequences $\mathbf{\pi_H}$
	and $\mathbf{\pi_L}$. This contradicts our assumption that $\mathbf{\pi}$ is forcibly connected.
	
	If both $\mathbf{\pi_H}$ and $\mathbf{\pi_L}$ have odd sums (we cannot have one of them having even sum and the other having
	odd sum since the sum of all the terms of $\mathbf{\pi}$ is even), then
	it must be the case that both $\mathbf{\pi_H}$ and $\mathbf{\pi_L}$ have an odd number of odd terms. Construct two new sequences
	$\mathbf{\pi'_H}$ and $\mathbf{\pi'_L}$
	from $\mathbf{\pi_H}$ and $\mathbf{\pi_L}$ by removing the largest odd term from $\mathbf{\pi_L}$ and adding it to $\mathbf{\pi_H}$.
	Now clearly $\mathbf{\pi'_H}$ and $\mathbf{\pi'_L}$ is a decomposition of $\mathbf{\pi}$ into two sub sequences of length $n/2+1$
	and $n/2-1$ respectively and both having even sums. Again they are guaranteed to be graphical degree sequences by the Nash-Williams condition
	using a similar argument as above, which contradicts the assumption that $\mathbf{\pi}$ is forcibly connected.
	
	The case for $n$ odd can be proved in a similar way. The conclusion that $M(n)$ cannot be of lower order
	than $\sqrt{n}$ then follows.
\end{proof}

We do not have any theory or algorithm to efficiently obtain $M(n)$ for any given $n$. Other than recording the minimum largest
term while enumerating all forcibly connected graphical degree sequences of length $n$,
a naive approach would be to let $\Delta$ start from 3 upward and test if there is a forcibly connected graphical degree sequence
of length $n$ with largest term $\Delta$ and stop incrementing $\Delta$ when we have found one.
Obviously this works but is not efficient. Any efficient algorithm for $M(n)$ might be worthwhile
to be added into Algorithm \ref{alg:fc} so that it can immediately return \textit{False} if $d_1<M(n)$.
However, while we conduct performance evaluations of Algorithm \ref{alg:fc} we do find that a random
graphical degree sequence of length $n$ with $d_1\le n/2$ most likely can be decided instantly by Algorithm
\ref{alg:fc}. Therefore we believe that an efficient algorithm for $M(n)$ will not help much on average.
We show the values of $M(n)$ based on our enumerative results in Table \ref{tab:Mn}. The fact that $M(15)=6$
agrees with the results of Table \ref{tab:enumDflargestpart(15)} where the counts $L_{15}[j]=0$ for all $j<6$.
As a side note, the minimum largest term in any potentially connected graphical sequence of length $n$ is clearly
2 since the degree sequence $2,2,\cdots,2$ ($n$ copies) is potentially connected while $1,1,\cdots,1$ ($n$ copies)
is not potentially connected.

\begin{table}[!htb]
	\centering
	\caption{Minimum largest term $M(n)$ of forcibly connected graphical sequences of length $n$.}
	\begin{tabular}[htbp]{|c||c|c|c|c|c|c|c|c|c|c|c|c|}
		\hline
		$n$ & 3 & 4 & 5 & 6 & 7 & 8 & 9 & 10 & 11 & 12 & 13 & 14 \\
		\hline
		$M(n)$ & 2 & 2 & 2 & 3 & 3 & 3 & 4 & 4 & 5 & 5 & 5 & 6 \\ \hline \hline
		$n$ & 15 & 16 & 17 & 18 & 19 & 20 & 21 & 22 & 23 & 24 & 25 & 26 \\
		\hline
		$M(n)$ & 6 & 6 & 7 & 7 & 7 & 7 & 8 & 8 & 8 & 8 & 8 & 9 \\ \hline
	\end{tabular}
	\label{tab:Mn}
\end{table}


We also investigated the number $g_f(n)$ of forcibly connected graphical partitions of a given even integer $n$. There is
a highly efficient Constant Amortized Time (CAT) algorithm of Barnes and Savage \cite{BarnesSavage1997}
to generate all graphical partitions of a given even $n$. And there are efficient counting algorithms of Barnes and
Savage \cite{BarnesSavage1995} and Kohnert \cite{Kohnert2004}
to count the number $g(n)$ of graphical partitions of even $n$ without generating them. It is known from Erd{\H{o}}s and Richmond
\cite{Erdos1993} that the number $g_c(n)$ of potentially connected graphical partitions of $n$ and $g(n)$ are of equivalent
order, i.e. $\lim_{n \to \infty} \frac{g_c(2n)}{g(2n)} = 1$. It is also known from Pittel \cite{Pittel1999}
that the proportion of graphical partitions among all partitions of an integer $n$ tends to 0. Although the order of the number
$p(n)$ of unrestricted partitions of $n$ is long known since Hardy and Ramanujan \cite{Hardy1918}, the exact asymptotic order
of $g(n)$ is still unknown. We know of no algorithm to count the number $g_f(n)$ of forcibly connected graphical
partitions of $n$ without generating them. Using a strategy similar to that we employed in computing $D_f(n)$,
we adapted the algorithm of Barnes and
Savage \cite{BarnesSavage1997} and incorporated the test of forcibly connectedness from Algorithm \ref{alg:fc}
and then count those that are forcibly connected. The growth of $g(n)$ is quick and we only have numerical results
of $g_f(n)$ for $n$ up to 170. The results together with the proportion
of them in all graphical partitions are listed in Table \ref{tab:enumgf(n)}. For the purpose of saving space we only
show the results in increments of 10 for $n$. From the table it seems
reasonable to conclude that the proportion $g_f(n)/g(n)$ will decrease when $n$ is beyond some small threshold
and it might tend to the limit 0.

\begin{table}[!htb]
	\centering
	\caption{Number of forcibly connected graphical partitions of $n$ and their proportions in all graphical partitions of $n$}
	\begin{tabular}[htbp]{|c||c|c|c|}
		\hline
		$n$ & $g(n)$ & $g_f(n)$ & $g_f(n)/g(n)$ \\
		\hline
		\hline
		10 & 17 & 8 & 0.470588 \\ \hline
		20 & 244 & 81 & 0.331967 \\ \hline
		30 & 2136 & 586 & 0.274345 \\ \hline
		40 & 14048 & 3308 & 0.235478 \\ \hline
		50 & 76104 & 15748 & 0.206927 \\ \hline
		60 & 357635 & 66843 & 0.186903 \\ \hline
		70 & 1503172 & 256347 & 0.170537 \\ \hline
		80 & 5777292 & 909945 & 0.157504 \\ \hline
		90 & 20614755 & 3026907 & 0.146832 \\ \hline
		100 & 69065657 & 9512939 & 0.137738 \\ \hline
		110 & 219186741 & 28504221 & 0.130045 \\ \hline
		120 & 663394137 & 81823499 & 0.123341 \\ \hline
		130 & 1925513465 & 226224550 & 0.117488 \\ \hline
		140 & 5383833857 & 604601758 & 0.112299 \\ \hline
		150 & 14555902348 & 1567370784 & 0.107679 \\ \hline
		160 & 38173235010 & 3951974440 & 0.103527 \\ \hline
		170 & 97368672089 & 9714690421 & 0.099772 \\ \hline
	\end{tabular}
	\label{tab:enumgf(n)}
\end{table}

\begin{table}[!htb]
	\centering
	\caption{Number of potentially (row $c_{20}[j]$) and forcibly (row $f_{20}[j]$) connected graphical partitions of 20 with given number of parts $j$.}
	\begin{tabular}[htbp]{|c||c|c|c|c|c|c|c|}
		\hline
		number of parts $j$ & 5 & 6 & 7 & 8 & 9 & 10 & 11 \\ \hline
		\hline
		$c_{20}[j]$ & 1 & 9 & 26 & 38 & 37 & 36 & 30 \\
		\hline
		$f_{20}[j]$ & 1 & 9 & 25 & 22 & 10 & 9 & 5 \\ \hline
	\end{tabular}
	\label{tab:enumgfparts(20)}
\end{table}

\begin{table}[!htb]
	\centering
	\caption{Number $l_{20}[j]$ of forcibly connected graphical partitions of 20 with given largest term $j$.}
	\begin{tabular}[htbp]{|c||c|c|c|c|c|c|c|c|}
		\hline
		largest part $j$ & 3 & 4 & 5 & 6 & 7 & 8 & 9 & 10 \\ \hline
		\hline
		$l_{20}[j]$ & 1 & 14 & 26 & 20 & 12 & 5 & 2 & 1 \\ \hline
	\end{tabular}
	\label{tab:enumgflargestpart(20)}
\end{table}

\begin{table}[!htb]
	\centering
	\caption{Minimum largest term $m(n)$ of forcibly connected graphical partitions of $n$.}
	\begin{tabular}[htbp]{|c||c|c|c|c|c|c|c|c|c|c|}
		\hline
		$n$ & 10 & 20 & 30 & 40 & 50 & 60 & 70 & 80 & 90 & 100 \\ \hline
		\hline
		$m(n)$ & 2 & 3 & 4 & 5 & 5 & 6 & 6 & 6 & 7 & 7  \\ \hline \hline
	\end{tabular}
	\label{tab:mn}
\end{table}

Like the situation for $D_f(n)$ the adapted algorithm from Barnes and Savage \cite{BarnesSavage1997}
to compute $g_f(n)$ actually generates all forcibly connected
graphical partitions of $n$ so it is trivial to also output the individual counts based on the number of parts or the largest part.
In Table \ref{tab:enumgfparts(20)} we show the individual counts of potentially and forcibly connected graphical partitions
of 20 based on the number of parts. Counts for the number of parts less than 5 or greater than 11 are not shown since those
counts are all 0. The ranges of the number of parts $j$ for which the number $c_n[j]$ of potentially connected graphical
partitions of $n$ with $j$ parts and the number $f_n[j]$ of forcibly connected graphical partitions of $n$
with $j$ parts are nonzero are exactly the same
based on Proposition \ref{thm:prop_potentialForciblyAgree}. The smallest number of parts $j$ for which $c_n[j]$ and
$f_n[j]$ are both nonzero is the smallest positive integer $t(n)$ such that $t(n)(t(n)-1)\ge n$ and this is also the smallest
number of parts for which a partition of $n$ with this many parts might be graphical. The largest number of parts $j$
for which $c_n[j]$ and $f_n[j]$ are both nonzero is $n/2+1$ based on the Wang and Cleitman characterization \cite{WangKleitman1973}.
In Table \ref{tab:enumgflargestpart(20)} we show the individual counts of forcibly connected graphical partitions of 20
based on the largest part. Counts for the largest part less than 3 or greater than 10 are not shown since those
counts are all 0. Clearly $n/2$ is the maximum largest part of any forcibly connected graphical partition of $n$
since $n/2,1,1,\cdots,1$ ($n/2$ copies of 1) is a forcibly connected graphical partition of $n$ and no graphical partition of $n$ has
its largest part greater than $n/2$. However, similar to the case of $M(n)$, the minimum largest part, $m(n)$,
of any forcibly connected graphical partition of $n$ does not seem to be easily obtainable. Clearly $m(n)$ grows
at most like $\sqrt{n}$ since for every large even $n$ it has a graphical partition with about $\sqrt{n}$ parts and all parts
about $\sqrt{n}-1$ and this graphical partition is forcibly connected. In Table \ref{tab:mn} we show several values of $m(n)$. They
are obtained while we exhaustively generate all graphical partitions of $n$ and keep a record of the minimum largest
part. The fact that $m(20)=3$ agrees with the results of Table \ref{tab:enumgflargestpart(20)} where $l_{20}[j]=0$
for all $j<3$. As a side note, the minimum
largest part of any potentially connected graphical partition of $n$ is clearly 2 since $2,2,\cdots,2$ ($n/2$ copies)
is a potentially connected graphical partition of $n$ while $1,1,\cdots,1$ ($n$ copies) is not.

\subsection{Questions and conjectures}
Based on the available enumerative results we ask the following questions and make certain conjectures:

1. What is the growth order of $D_f(n)$ relative to $D(n)$? Note that the exact asymptotic order of $D(n)$ is unknown yet.
(Some upper and lower bounds of $D(n)$ are known. See Burns \cite{Burns2007}).
We conjecture $\lim_{n \to \infty} \frac{D_f(n)}{D(n)} = 1$.
That is, almost all zero-free graphical degree sequences of length $n$ are forcibly connected. If this is true, then it is a stronger
result than the known result (see \cite{Wang2016}) $\lim_{n \to \infty} \frac{D_c(n)}{D(n)} = 1$ since $D_f(n)\le D_c(n)\le D(n)$.
Furthermore, we conjecture that $D_f(n)/D(n)$ is monotonously increasing when $n\ge 8$. Let $D_{c\_k}(n)$ and $D_{f\_k}(n)$
denote the number of potentially and forcibly $k$-connected graphical degree sequences of length $n$ respectively.
It is already known from \cite{Wang2016} that $\lim_{n \to \infty} \frac{D_{c\_k}(n)}{D(n)} \neq 1$ when $k\geq 2$. Clearly we
also have $\lim_{n \to \infty} \frac{D_{f\_k}(n)}{D(n)} \neq 1$ when $k\geq 2$.
What can be said about the relative orders of $D_{c\_k}(n)$, $D_{f\_k}(n)$ and $D(n)$ when $k\geq 2$?

2. What is the growth order of $g_f(2n)$ relative to $g(2n)$? Note that the exact asymptotic order of $g(2n)$ is unknown yet.
We conjecture $\lim_{n \to \infty} \frac{g_f(2n)}{g(2n)} = 0$. That is, almost none of the graphical partitions of $2n$
are forcibly connected. Furthermore, we conjecture that $g_f(2n)/g(2n)$ is monotonously decreasing when $n\ge 5$.
Let $g_{c\_k}(n)$ and $g_{f\_k}(n)$
denote the number of potentially and forcibly $k$-connected graphical partitions of $n$ respectively.
What can be said about the relative orders of $g_{c\_k}(n)$, $g_{f\_k}(n)$ and $g(n)$ when $k\geq 2$?

3. We conjecture that the numbers of forcibly connected graphical partitions
of $N$ with exactly $n$ parts, when $N$ runs through $2n-2,2n,\cdots,n(n-1)$, give a unimodal sequence.

4. Let $t(n)$ be the smallest positive integer such that $t(n)(t(n)-1)\ge n$.
We conjecture that the numbers of forcibly connected graphical partitions
of $n$ with $j$ parts, when $j$ runs through $t(n), t(n)+1, \cdots, n/2+1$, give a unimodal sequence.

5. What is the growth order of $M(n)$, the minimum largest term in any forcibly connected graphical sequence of
length $n$? Is there a constant $C>0$ such that $\lim_{n \to \infty} \frac{M(n)}{n} = C$?
Is there an efficient algorithm to compute $M(n)$?

6. What is the growth order of $m(n)$, the minimum largest term in any forcibly connected graphical partition of
$n$? Is there a constant $C>0$ such that $\lim_{n \to \infty} \frac{m(n)}{\sqrt{n}} = C$?
Is there an efficient algorithm to compute $m(n)$?

7. We conjecture that the numbers of forcibly connected graphical partitions of an even $n$
with the largest part exactly $\Delta$, when $\Delta$ runs through $m(n),m(n)+1,\cdots,n/2$, give a unimodal sequence.

8. We showed all these decision problems to test whether a given graphical degree sequence is forcibly $k$-connected
to be in co-NP for fixed $k\ge 1$. Are they co-NP-hard? Is the decision
problem for $k+1$ inherently harder than for $k$?

\section{Conclusions}
In this paper we presented an efficient algorithm to test whether a given graphical degree sequence is forcibly
connected or not and its extensions to test forcibly $k$-connectedness of graphical degree sequences for fixed $k\ge 2$.
Through performance evaluations on a wide range of long random graphical degree sequences we
demonstrate its average case efficiency and we believe that it runs in polynomial time on average.
We then incorporated this testing algorithm into existing algorithms that enumerate zero-free graphical
degree sequences of length $n$ and graphical partitions of an even integer $n$ to obtain some enumerative results
about the number of forcibly connected graphical degree sequences of length $n$ and forcibly connected graphical
partitions of $n$. We proved some simple observations related to the available numerical results and
made several conjectures. We are excited that there are a lot for further research in this direction.

\bibliographystyle{plain}
\bibliography{testFC}

\end{document}